\documentclass[a4paper,12pt]{amsart}
\usepackage{amsfonts}
\usepackage{mathrsfs}
\usepackage{amsmath,amssymb,latexsym,amsfonts,amscd}
\usepackage[small,nohug,UglyObsolete]{diagrams}
\diagramstyle[labelstyle=\scriptstyle]

\newcommand\dep{{\text{\rm dep}}}

\newcommand\cX{{\mathcal{X}}}

\newcommand\bC{{\mathbb C}}

\newcommand\bP{{\mathbb P}}
\newcommand\bQ{{\mathbb Q}}

\newcommand\bZ{{\mathbb Z}}

\newtheorem{thm}{Theorem}[section]

\newtheorem{cor}[thm]{Corollary}
\newtheorem{prop}[thm]{Proposition}

\theoremstyle{definition}
\newtheorem{defn}[thm]{Definition}

\newtheorem{rem}[thm]{Remark}
\theoremstyle{remark}

\title{Birational maps of $3$-folds}
\author{Jungkai Alfred Chen}

\address{\rm National Center for Theoretical Sciences (NCTS/TPE) and Department of Mathematics, National Taiwan University, Taipei,
106, Taiwan} \email{jkchen@ntu.edu.tw}

\thanks{The  author was partially supported by  NCTS/TPE
and  National Science Council of Taiwan.
  We are indebted to Cascini, Hacon, Hayakawa, Kawakita, Kawamata, Koll\'ar and  Mori for many useful
  discussion.
Some of this work was done during  visits of the  author to RIMS
and Imperial College London. The  author would like to thank both
institutes for their hospitality.}

\begin{document}
\begin{abstract}
We show that $3$-fold terminal flips and divisorial contractions
may be factored into a sequence of  flops, blow-downs to a smooth
curve in a smooth $3$-fold or divisorial contractions to  points
with minimal discrepancies.
\end{abstract}
\maketitle

\section{Introduction} In birational geometry, one of the main task is to find a good model inside a birational equivalence class and study the geometry of models. This goal can be achieved by minimal model program. The minimal model conjecture asserts that for any given nonsingular or mildly singular projective variety, there exists a minimal model or a Mori fiber space after a sequence of flips and divisorial contractions. Moreover, different minimal models are connected by a sequence of flops. Therefore divisorial contractions, flips and flops  are the elementary birational
maps of the minimal model program.

Together with some recent advances on geometry of $3$-folds, for
example, $m$-th canonical maps is birational for $m \ge 73$ and
the canonical volume $\ge \frac{1}{2660}$ (cf. \cite{ExplicitI,
ExplicitII}), one might hope to build up an explicit
classification theory for $3$-folds similar to the theory of
surfaces by using the minimal model program explicitly. To this
end,  it is thus natural to ask how explicit do we know about
birational maps in three-dimensional  minimal model program. Even
though the minimal model program for $3$-folds was "proved" in
more than twenty years ago by Mori and others, the more detailed
and explicit description of birational maps in $3$-dimensional
minimal model program was available only quite recently and not
completely satisfactory. To give a quick tour of known results:
Mori and then Cutkosky classified birational maps from a
nonsingular and Gorenstein $3$-fold respectively \cite{Mo82, Cu},
and Tziolas has a   series of works on divisorial contractions to
curves passing through Gorenstein singularities (cf. \cite{Tz03,
Tz05, Tz09}). Divisorial contractions to points are probably most
well-understood mainly thanks to the work of Kawamata, Hayakawa,
Markushevich and Kawakita (cf. \cite{Ka, HaI, HaII, Ha1, Ma, Kk01,
Kk02, GE, Kk05,  Kk11}). Also, the structure of flops are studied
in Koll\'ar's article \cite{Ko89}. Flips are still quite
mysterious except for some examples in \cite{KM92, Br} and toric
flips \cite{Re}.

Instead of classifying birational maps completely,
we  work on the problem to factorize birational maps into a composition of simplest ones.
Such factorization can be very useful for comparing various invariants between birational models.
It is also useful in classifying birational maps. In the previous joint work with Christopher Hacon \cite{CH},
we are able to factorize flips and divisorial contractions to curves.
Our previous work \cite{Pisa} factorizes divisorial contractions to a point of index $r>1$ with non-minimal discrepancy $\frac{a}{r}> \frac{1}{r}$.
The purpose of this note is to
show that one can factor threefold birational maps in minimal
model program into some simple and explicit ones by combing previous work \cite{CH, Pisa} and considering divisorial contraction to a point of index $r=1$.

\begin{defn}
A birational map $f: X \dashrightarrow Y$ is {\it factorizable} if
it admits a factorization into a sequence of birational maps:
$$  X=X_0 \dashrightarrow X_1 \dashrightarrow \ldots
\dashrightarrow X_n=Y,$$ such that each map $X_{i-1}
\dashrightarrow X_{i}$ is one of the following
\begin{enumerate}
\item a divisorial contraction (or its inverse) to a point $P_i
\in X_i$ of index $r_i \ge 1$ with minimal discrepancy;

\item a blowup along a smooth curve in a smooth neighborhood;

\item a flop.
\end{enumerate}
\end{defn}

\begin{thm}[=Main Theorem]\label{main}
A three dimensional divisorial contraction $f: X \to W$ (resp.
flip $\phi: X \dashrightarrow X^+$) is factorizable.
\end{thm}

\begin{rem}
Given a divisorial contraction to a point $f: X \to W \ni P$ with
exceptional divisor $E$. Then we can write $K_X = f^*K_W + aE$. We
say that the contraction $f$ has discrepancy $a$.

Given   $P$  a terminal singularity of index $r$, then the minimal
discrepancy among all divisorial contractions to $P$ is
$\frac{1}{r}$ by \cite{Ma} and \cite{Ka}. If $P \in W$ is a
nonsingular point, then the minimal discrepancy among all
contractions to $P$ is $2$ by \cite{Kk01}.
\end{rem}

The key observation is that for any   complicated divisorial
contraction $X \to W$ (resp. flip $X \dashrightarrow X^+$), there
exists  singular points of index $r>1$ on $X$. By choosing  $Q \in
X$ a point of higher index and choosing a divisorial contraction
$Y \to X$ to the point $Q \in X$ with discrepancy $\frac{1}{r}$,
we shall prove that there exists a diagram of birational maps:

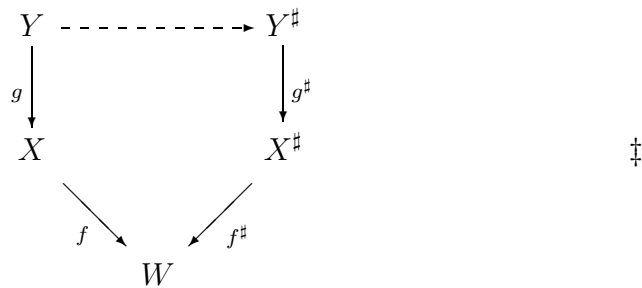
\begin{equation*}
\begin{diagram}
Y     &   & \rDashto &    &  Y^\sharp   \\
\dTo^{g} &       &      &   &  \dTo_{g^\sharp}      \\
 X       &        &      &  &  X^\sharp  \\
           & \rdTo_{f}  &      & \ldTo_{f^\sharp}  &        \\
           &        &  W   &        &
\end{diagram}\eqno{\ddagger}
\end{equation*}
where $Y \dashrightarrow Y^\sharp$ consists of a sequence of flips
and flops, $g^\sharp$ is a divisorial contraction, and $f^\sharp$
is also a divisorial contraction (resp. $f^\sharp $ is the flipped
map). We thus call that  $\ddagger$ is a {\it factoring diagram
for $X \to W$ (resp. $X \dashrightarrow X^+$)}.

If $f$ is a weighted blowup, then the factoring diagram can be
constructed by using toric geometry and a few computation. This
was the approach in \cite{Pisa}. In the remaining divisorial
contractions which are not known to be weighted blowups, usually
there is a unique non-Gorenstein singularity $P \in X$ of pretty
high index. By choosing a divisorial contraction $g: Y \to X$ with
minimal discrepancy, one can verify that there is only a little
change in the intersections. Computation shows that $-K_{Y/W}$ is
nef and one can thus play the so-called $2$-ray game to obtain the
factoring diagram.

Moreover, by considering depth (cf. \cite{CH}) and discrepancy, one sees that  $Y, Y^\sharp, X^\sharp$ has milder singularities in some sense. Our result then follows by induction using the factoring diagram.

\section{notations and preliminary}

We always work on complex threefolds with  $\bQ$-factorial
 singularities (unless the image of flipping contraction).
Recall that threefold terminal singularities of index $1$ are
isolated $cDV$ points and terminal singularities of index $r>1$
are classified by Mori (cf. \cite{Mo85}).

This work can be considered as a continuation of our previous work
\cite{CH, Pisa}. We usually adapt the constructions and notations
there.

Given a threefold terminal singularity $P \in X$ of index $r>1$,
by \cite{HaI, HaII}, there exists a partial resolution $$ X_n \to
\ldots \to X_1 \to X_0=X \eqno{\dagger}$$ such that $X_n$ has
Gorenstein singularities and each $X_{i+1} \to X_{i}$ is a
divisorial contraction to a point $P_i \in X_i$ of index $r_i
>1$ with discrepancy $\frac{1}{r_i}$.  The definition of depth was  introduced in
\cite{CH}.
$$ \dep(P \in X):=\min\{n| X_n \to X \ni P \text{ is a partial
resolution as above}\}.$$

The following properties for depth are useful.
\begin{prop} \label{depth}
The following properties for depth holds.
\begin{enumerate}
\item
Let $\phi: X \dashrightarrow X^+$ be a flip (resp. flop), then $\dep(X) > \dep(X^+)$ (resp. $\dep(X)=\dep(X^+)$).

\item
Let $f: X \to W$ be a divisorial contraction to a curve, then $\dep(X) \ge \dep(W)$. Equality holds if and only if $\dep(X)=\dep(W)=0$.

\item Let $f: X \to W$ be a divisorial contraction to a point, then $\dep(X)+1 \ge \dep(W)$.
\end{enumerate}
\end{prop}

\begin{proof}
All the statements were proved in \cite[Proposition 2.15, 3.5,
3.6]{CH} except the strict inequality for divisorial contractions
to curves when $\dep(X) >0$. Recall that by \cite{CH}, there is a
factoring diagram
\begin{diagram}
Y     &   & \rDashto &    &  Y^\sharp   \\
\dTo^{g} &       &      &   &  \dTo_{g^\sharp}      \\
 X       &        &      &  &  X^\sharp  \\
           & \rdTo_{f}  &      & \ldTo_{f^\sharp}  &        \\
           &        &  W   &        &
\end{diagram}
such that $Y \to X$ is a divisorial contraction to a higher index
point $Q \in X$  with $r(Q \in X) >1$ and  discrepancy
$\frac{1}{r}$, and $\dep(Y)=\dep(X)-1$. Moreover, $Y^\sharp \to
X^\sharp$ is a divisorial contraction to a curve and $X^\sharp \to
W$ is a divisorial contraction to a point.

If $\dep(X)=1$ and suppose that $\dep(W) >0$, then $\dep(W)=1$ for $\dep(X) \ge \dep(W)$ by
\cite[Proposition 3.6]{CH}. Then by definition of depth, it is
easy to see tat $W$ has only one quotient singularity of type
$\frac{1}{2}(1,1,1)$. It follows that $X \to W$ is the weighted
blowup with weights $v=\frac{1}{2}(1,1,1)$ by \cite{Ka}, which is
absurd.
We thus conclude that $\dep(W)=0 < \dep(X)$.

In general $\dep(X)=d>1$, then $\dep(Y^\sharp) \le \dep(Y)=d-1$.
By induction hypothesis, one has $\dep(X^\sharp) <\dep(Y^\sharp)
\le d-1$. It follows that $\dep(W) \le \dep(X^\sharp)+1 <d$ by
\cite[Proposition 2.15]{CH}.
\end{proof}

\section{divisorial contractions to curves}
The purpose of this section is to factorize threefold divisorial contraction to curves. Let $f: X \to W$  be a divisorial contraction to a curve $\Gamma \subset W$ such that $X$ has at worst terminal Gorenstein singularities. By \cite{Mo82, Cu}, it is known that $W$ is smooth near $\Gamma$ and $\Gamma \subset W$ is a lci curve. Moreover, $f$ is the blowup along $\Gamma$.

If $\Gamma$ is a nonsingular curve, then $f: X \to W$ is nothing
but the blowups along $\Gamma$. If the curve $\Gamma$ is singular
at $o$, then one can factorize the divisorial contraction $f: X
\to W$ by the following diagram.

\begin{prop} \label{lci} Keep the notation as above.   Then there is a factoring diagram as $\ddagger$ of birational maps
such that
\begin{enumerate}
\item $Y \dashrightarrow Y^\sharp$ consists of a sequence of flops;

    \item  $f^\sharp$ is the blowup along $o \in W$;
     \item $g^\sharp$ is the blowup of $X^\sharp$ along $\Gamma^\sharp$, where $\Gamma^\sharp$ is  the proper transform of $\Gamma$ in $X^\sharp$;
    \item the induced map $\Gamma^\sharp \to \Gamma$ is isomorphic to the blowup of $\Gamma$ over $o$;
    \item $g$ is a divisorial contraction to a singular point $Q \in X$ of type $cA$ with discrepancy $1$.
\end{enumerate}
\end{prop}

\begin{proof}
 Recall that a weighted blowup for a toric variety can be obtained by subdivision along a primitive vector $v$ and the exceptional divisor is the divisor corresponding to the vector $v$. Also a weighted blowup for a complete intersection in a toric variety is considered to be  the induced map from its proper transform.  For detailed description, please see \cite{Pisa} for example.

 By shrinking $W$, we
may assume that $X$ is an open subset in $\bC^3$,
$\Gamma=(x_3=h(x_1,x_2)=0) \subset W \subset \bC^3$ and $o \in
\Gamma$ is the only singular point of $\Gamma$. Let
$\tau:=\text{mult}_o h(x_1,x_2) \ge 2$.

We consider towers of weighted blowups $\cX_2
\stackrel{\pi_g}{\to} \cX_1 \stackrel{\pi_f}{\to} \cX_0$, where
$\cX_0 = \bC^4$, $\pi_g$ (resp $\pi_f$) are weighted blowup along
the vector  $v_2=(1,1,\tau-1,\tau)$ (resp $v_1=(0,0,1,1)$). More
explicitly,  $\pi_f$ is the blowup of $\cX_0$ along
$\Sigma:=(x_3=x_4=0)$ and $\cX_1$ is covered by two affine pieces
$U_3 \cup U_4$. One sees also that $\pi_g$ is the weighted blowup
over the origin of $U_3$ with  weights $(1,1,\tau-1,1)$.

%
We may consider an embedding $W \hookrightarrow \bC^4$ that
$W=(x_4-h(x_1,x_2)=0)$. Now $\Gamma=W \cap \Sigma$ and the given
divisorial contraction $f: X \to W$ coincides with the induced map
${\pi_f}_{|_X}$. On $X$, there is a unique singularity $Q_3$ of
$cA$ type locally given by $x_3x_4-h(x_1x_2)=0$. Moreover, let $Y$
be the proper transform of $X$ in $\cX_2$. The induced map $g: Y
\to X$, which is the weighted blowup with weights $(1,1,\tau-1,1)$
over $Q_3$, is clearly a divisorial contraction to $Q_3$ with
discrepancy $1$.

Let $l:=f^{-1}(o) \cong \bP^1$, $l_Y$ be the proper transform of
$l$ in $Y$. It is easy to see that $l \cdot K_X=-1$ and $l_Y \cdot
K_Y= 0$. We remark that there is only one singularity on $Y$, which is a quotient singularity of index $\tau-1$ and does not contained in $l_Y$. By the same argument in
\cite[Theorem 3.3]{CH}, one has a factoring diagram as $\ddagger$
and a tower of divisorial contractions $Y^\sharp \to X^\sharp \to
W$.



On the other hand, we may consider $Y' \to X' \to W$ by weighted
blowup with vector $v_2=(1,1,\tau-1,\tau)$ and then $v_1=(0,0,1,1)$. By
the same argument as in \cite[Theorem 2.7]{Pisa}, the tower $Y'
\to X' \to W$ is isomorphic to $Y^\sharp \to X^\sharp \to W$.

Let $\Gamma'$ be the proper transform of $\Gamma$ in $X'$.
Computation shows that
both $X' \to W$ and $\Gamma' \to \Gamma$ are isomorphic to the
blowup over $o \in W$ and $o \in \Gamma$. Moreover, $Y' \to X'$ is
the blowup along $\Gamma'$.
Since the only singularity on $X'$ is a quotient singularity $Q'_3$ of index $\tau-1$ and $\Gamma'$ does not contains $Q'_3$. Therefore $\dep(X')=\dep(Y')=\tau-1=\dep(Y)$. It follows that $Y \dashrightarrow Y'$ consists of a sequence of flops only by Proposition \ref{depth}.
This completes the proof.
\end{proof}

By the above diagram  successive over the  singular points of $\Gamma$, one get the following consequence immediately.

\begin{cor}
Let  $f: X \to W$  be a divisorial contraction to a curve $\Gamma \subset W$ such that $X$ has at worst terminal Gorenstein singularities. The $f: X \to W$ is factorizable.
\end{cor}

\section{divisorial contractions to points}

Divisorial contractions to points was intensively studied by
Kawamata, Hayakawa, and Kawakita \cite{Ka, HaI, HaII, Kk01, Kk02,
GE, Kk05, Kk11}. We give a brief summary of  the known
classification.
\begin{itemize}
\item If $f: X \to W \ni P$ is a divisorial contraction to a point $P \in W$ of index $r>1$ with discrepancy $\frac{a}{r} \ge \frac{1}{r}$, then $f$ is completely classified. Any of these can be realized as a weighted blowup explicitly (cf. \cite{Ka, HaI, HaII, Kk05, Kk11}).

\item If $f: X \to W \ni P$ is a divisorial contraction to a point $P \in W$ of index $r=1$ with discrepancy $1$.

\item If $f: X \to W \ni P$ is a divisorial contraction to a point $P \in W$ of index $r=1$ with discrepancy $a>1$, then $f$ is one of following cases in Table A.
\end{itemize}

\centerline{\textbf{Table A.}} \smallskip
\begin{tabular}{|l|l|l|c|l|}
\hline
type & $P \in W$ & discrepancy & w. blowup & reference \\
 \hline
 Ia & nonsingular & $a+b$ & Yes &  \cite[Theorem 1.1]{Kk01} \\
 \hline
 Ib & $cA$ & $a \ge 1$ & Yes & \cite[Theorem 1.2.i]{Kk05} \\
  \hline
 Ic & $cD$ & $a>1$,  odd & Yes & \cite[Theorem 1.2.ii.a]{Kk05} \\
  \hline
  Id & $cD$ & $a>1$ & Yes & \cite[Theorem 1.2.ii.b]{Kk05} \\
  \hline
 IIa & $cA_1$ & $4$ & Yes & \cite[Theorem 2.5]{Kk02} \\
  \hline

IIb & $cE_{7,8}$ & $2$ & ? & \cite[Table 3, e9]{Kk05} \\
  \hline
IIc & $cE_{7}$ & $2$ & ? & \cite[Table 3, e5]{Kk05} \\
  \hline
IId & $cA_2, cD, cE_6$ & $3$ & ? & \cite[Table 3, e3]{Kk05} \\
  \hline
IIe & $cD, cE_{6,7}$ & $2$ & ? & \cite[Table 3, e2]{Kk05} \\
  \hline
IIf & $cD$ & $2$ & ? & \cite[Table 3, e1]{Kk05} \\
  \hline
IIg & $cD$ & $4$ & ? & \cite[Table 3, e1]{Kk05} \\
  \hline

\end{tabular}

\bigskip

The purpose of this section is to construct a factoring diagram
$\ddagger$ for divisorial contraction  with non-minimal
discrepancy $a>1$ as listed in Table A. Given  a divisorial
contraction with non-minimal discrepancy $f: X \to W \ni P$. Let
$E$ be its exceptional divisor. By the classification of
\cite{Mo82},\cite{Cu}, $X$ can not be Gorenstein. We will pick a
point $Q \in X$ of index $p >1$.

For  any divisor $D$ on $X$ passing through $Q$, we set
$D_{W}=f^{}_*D$, $D_{Y}=g^{-1}_*D$ to be the proper transform  of
$D$ on $W,Y$ respectively. Let  $E_Y$ denotes  the proper
transform of $E$ on $Y$. We have
$$f^*D_{W}=D+ \frac{c_0}{n} E, \quad g^*D =
D_{Y}+\frac{q_0}{p} F, \quad g^*E=E_Y+\frac{\frak q}{p}F$$ for some $c_0, q_0,  \frak q \in \bZ_{>0}$.

\begin{prop}{\cite[Proposition 2.4]{Pisa}} \label{nef}  Let  $f: X \to W$ be a divisorial contraction to a point $P \in
W$ of index $n$ with discrepancy $\frac{a}{n} $ and
$E$ the exceptional divisor of $f$. Let $g: Y \to X $ be a divisorial contraction to a point $Q \in E$
of index $p$ with discrepancy $\frac{b}{p} $. Suppose that there is a  divisor $D$ on $X$ such that $D \cap E$ is irreducible.
  Then $-K_{Y/W}$
is nef if  the following inequalities holds:
$$ \left\{
\begin{array}{l}
T(f,g,D):=\frac{-ac_0}{n^2}
E^3+\frac{q_0\frak q b}{p^3}F^3 \le  0;\\
bc_0-a q_0 \le 0.
\end{array}
\right. \eqno{\dagger}$$
\end{prop}

In \cite[Theorem 1.5]{Kk11}, Kawakita give an affirmative answer to the General Elephant Conjecture. In particular, let $f: X \to W$ be a divisorial contraction, then a general element $S_X \in |-K_X|$ is normal and  has only Du Val singularities.

\begin{prop}{\cite[Proposition 2.5]{Pisa}} \label{2ray} Let $f: X \to W$ be a divisorial contraction to a point with exceptional divisors $E$ and $g: Y \to X$ be a divisorial contraction to a point $ Q \in E \subset X$  of index $p$ with discrepancy $\frac{1}{p}$. Let $F$ be the exceptional divisor of $g$.  Suppose that $-K_{Y/W}$ is nef and there is an irreducible curve $l \subset S_X \cap E$ such that $l_Y \cdot K_Y <0$, then we have the factoring diagram $\ddagger$ such that
\begin{enumerate}
 \item $\phi: Y \dashrightarrow Y^\sharp$ is a sequence of flips and flops (or just the identity map);
  \item $g^\sharp$ is a divisorial contraction contracting $E_{Z^\sharp}$;
  \item $f^\sharp$ is a divisorial contraction contracting $F_{Y^\sharp}$ to the point $P \in W$.
 \end{enumerate}
\end{prop}

We will need the following variant. The proof is almost the same as {\cite[Corolary 2.6]{Pisa}}.

\begin{cor} \label{2ray2} Let $f: X \to W$ be a divisorial contraction to a point with exceptional divisors $E$ and $g: Y \to X$ be a divisorial contraction to a point $ Q \in E \subset X$  of index $p$ with discrepancy $\frac{1}{p}$. Let $F$ be the exceptional divisor of $g$.  Suppose that $l_Y \cdot K_Y \le 0$ for any irreducible curve $l \subset S_X \cap E$ and $T(f,g):=\frac{-a^2}{n^2}E^3+\frac{\frak{q}}{p^3}F^3<0$. Then we have a factoring diagram $\ddagger$ as in Proposition \ref{2ray}.
\end{cor}

An immediate but useful consequence is the following:

\begin{cor} \label{2ray3}
Keep the notation as in Corollary \ref{2ray2}. Suppose that $Q
\in E$ is the only non-Gorenstein point on $E$, which is of index
$p>1$. Suppose furthermore that $ \frac{\frak{q}}{p^3}F^3 <
\frac{1}{p}$. Then there exists a factoring diagram $\ddagger$ as in Proposition \ref{2ray}.
\end{cor}

\begin{proof}
Suppose that $[S_X \cap E]=[\sum c_i l_i]$ as $1$-cycle for some
$c_i \in \bZ_{>0}$.
Note that $l_{i,Y} \cdot K_Y \ge l_{i} \cdot K_X$ for all $i$.
Hence for all $i$,
$$ \begin{array}{ll} l_{i,Y} \cdot K_Y & = l_{i} \cdot K_X + (l_{i,Y} \cdot K_Y - l_{i} \cdot K_X )\\
                                      & \le l_{i} \cdot K_X + \sum_i (l_{i,Y} \cdot K_Y - l_{i} \cdot K_X)\\
                                      & \le \frac{-1}{p} + \frac{\frak{q}}{p^3} F^3 <0. \end{array}.$$
By Corollary \ref{2ray2}, there exists a factoring diagram.
\end{proof}

We remark that once there is a factoring diagram, then the induced
map $f^\sharp: X^\sharp \to W$ is a divisorial contraction to $P
\in W$ with exceptional divisor $F_{X^\sharp}$ and discrepancy
$\frak{a}:= \frac{ a \frak q +n}{p} \in \bZ_{>0}$.

We now study the divisorial contraction to a Gorenstein point with
non-minimal discrepancies case by case (cf. Table A).

\noindent {\bf Case Ia.} Suppose that $P \in W $ is nonsingular. \\
By \cite{Kk01},  $f$ is the weighted blowup of weight $(1,m,n)$
with $(m,n)=1$,  $1< m < n$, and the discrepancy is $a=m+n$. 

On $X$,  the highest index point, say $Q$,  is a terminal quotient
singularity of type $\frac{1}{n}(1, m, -1)$. Let $g: Y \to X$ be
the Kawamata blowup, which is the weighted blowup of weights
$\frac{1}{n}(t, 1, n-t)$, where $t$ is the minimal positive
integer satisfying $mt=ns+1$. Clearly $t<n, s<m$.

Pick $D=f^{-1}_* {\rm div}(x_2)$. Then $l=D \cap E$  is clearly
irreducible. Since $c_0=m, q_0=1$ and $\frak q=n-t$, one has $$
T(f,g,D) = -\frac{m+n}{n} + \frac{1}{nt} <0.$$ Hence we have the
factoring diagram by Proposition \ref{2ray}. By Theorem 2.7 of
\cite{Pisa}, one sees that both $f^\sharp, g^\sharp$ are weighted
blowups. The factoring diagram indeed fits into the following
diagram.

$$\begin{CD}
Y @>{\dashrightarrow}>> Y^\sharp \\
@V{\frac{1}{n}}V{wt=w_2}V @V{{s+t}}V{wt=w'_2}V \\
Q_3 \in X @.  X^\sharp \ni Q^\sharp_1 \\
@V{m+n}V{wt=w_1  }V @V{{m+n-s-t}}V{wt=w'_1}V\\
W @>=>> W
\end{CD}$$
where
$$ \begin{array}{ll} w_1=(1,m,n), & w'_1=(1, m-s,n-t), \\
  w_2=\frac{1}{n}(t, 1, n-t),  & w'_2=(1, s, t). \end{array}$$

\noindent
{\bf Case Ib.} This contraction is described in \cite[Theorem 1.2.i]{Kk01}. In fact, the factoring diagram 
is described in \cite[Subsection 3.5]{Pisa} with $n=1$.
We give a brief review for reader's convenience. The equation of
$P \in W$ is given by $$ \varphi: x_1x_2 +g(x_3,x_4)=0 \subset
\bC^4.$$ The map $f$ is given by weighted blowup with weight
$v_1=(r_1,r_2,a,1)$. We may write $r_1+r_2=da$ for some $d>0$ with
the term $x_3^{d} \in \varphi$. Moreover, $(a,r_1)=(a,r_2)=1$.
Hence, there exist $0 < s_i^* < r_i$ and  $0<a_i<a$ so that

$$\left\{ \begin{array}{l}
1+a_1r_1=s_1^* a;\\
1+a_2r_2=s_2^* a.\\
\end{array} \right.
$$

 Note that
$as_2^*=1+a_2 r_2=1+ a_2 (ad-r_1)$. Therefore,
$ a( s_2^* - a_2 d)= 1 -a_2 r_1.$
By $(a,r_1)=1$ and comparing it with $a s_1^*=1+a_1 r_1$, we
have $a_1=-a_2+ ta$ for some $t \in \bZ$. Since $ 0<
a_1+a_2 <2a$, it follows that $a_1+a_2=a$.

 Suppose that $r_1>1$.
 We have the
 following factoring diagram.
$$\begin{CD}
Y @>{\dashrightarrow}>> Y^\sharp \\
@V{\frac{1}{r_1}}V{wt=w_2}V @V{{a_1}}V{wt=w'_2}V \\
Q_1 \in X @.  X^\sharp \ni Q^\sharp_4 \\
@V{a}V{wt=w_1  }V @V{{a_2}}V{wt=w'_1}V\\
W @>=>> W
\end{CD}$$
where
$$ \begin{array}{ll} w_1=(r_1,r_2,a,1), & w'_1=(r_1-s_1^*,r_2-a_1d+s_1^*,a_2,1) \\
  w_2=\frac{1}{r_1}(r_1-s_1^*,d,1,s_1^*),  & w'_2=(s_1^*, a_1 d-s_1^*,a_1,1). \end{array}$$


 Suppose that $r_2>1$.
We have the
 following factoring diagram.
$$\begin{CD}
Y @>{\dashrightarrow}>> Y^\sharp \\
@V{\frac{1}{r_2}}V{wt=w_2}V @V{{a_2}}V{wt=w'_2}V \\
Q_2 \in X @.  X^\sharp \ni Q^\sharp_4 \\
@V{{a}}V{wt=w_1  }V @V{{a_1}}V{wt=w'_1}V\\
W @>=>> W
\end{CD}$$
where
$$ \begin{array}{ll} w_1=(r_1,r_2,a,1), & w'_1=(r_1+s_2^*-a_2d,r_2-s_2^*,a_1,1) \\
  w_2=\frac{1}{r_2}(d, r_2-s_2^*,1,s_2^*),  & w'_2=(a_2 d-s_2^*,s_2^*, a_2,1). \end{array}$$

\noindent {\bf Case Ic.} 
This contraction is described in \cite[Theorem 1.2.ii.a]{Kk01} and
the discussion is parallel to the that in \cite[Subsection
3.2]{Pisa}. The local equation of $P \in W$ is given by
$$(\varphi:x_1^2+x_2^2 x_4+x_1q(x_3^2,x_4)+\lambda x_2x_3^{2}+\mu x_3^3+p(x_2,x_3,x_4)=0) \subset \bC^4,$$
$f$ is the weighted blowup with
weights $v_1=(r+1,r,a,1)$, $2r+1=ad$ and both $a,d$
are odd. Notice that $wt_{v_1}(\varphi)=2r+1$ and  we have that $x_3^{d} \in p(x_2,x_3,x_4)$ otherwise $Q_3 \in X $ is singular of index $a$.

There are two quotient singularities $Q_1,Q_2$ of index $r+1,r$
respectively.
We  take $g: Y \to X$  the   weighted blowup with weights $w_2=\frac{1}{r}(d,r-d,1,d)$  over $Q_2$.
 Then
$$E^3=\frac{2r+1}{ar(r+1)}, \quad F^3= \frac{r^2}{d(r-d)}, \quad \frak {q}=r-d,\quad \frak {a}=a-2.$$

In this case,  we pick  $S=f^{-1}_* \text{div}({x}_3) \in
|-K_X|$, then $S \cap E$ is irreducible.
Now $$T(f,g)=\frac{1}{r}(-\frac{a(2r+1)}{r+1}+\frac{1}{d}) <0.$$
Therefore there exists a factoring diagram  by
Proposition \ref{2ray}.

$$\begin{CD}
Y @>{\dashrightarrow}>> Y^\sharp \\
@V{\frac{1}{r}}V{wt=w_2}V @V{2}V{wt=w'_2}V \\
Q_2 \in X @.  X^\sharp \ni Q^\sharp_4 \\
@V{{a}}V{wt=w_1  }V @V{a-2}V{wt=w'_1}V\\
W @>=>> W
\end{CD}$$

where $$ \begin{array}{ll} w_1=v_1=(r+1,r,a,1), & w'_1=v_2=(r+1-d,r-d,a-2,1), \\
  w_2=\frac{1}{r}(d,r-d,1,d),  & w'_2=(d,d,2,1). \end{array}$$

\noindent {\bf Case Id.} In the case (1.2.ii.b), the local equation of $P \in W$ is given
by
$$(P \in W) \cong  o \in \left( \begin{array}{l}
\varphi_1: x_1^2+x_2x_5+p(x_2,x_3,x_4)=0 \\
\varphi_2: x_2 x_4 +x_3^{d} + q(x_3,x_4) x_4 + x_5 =0
\end{array} \right) \subset \bC^5,$$  $f$  is a weighted blowup
with weights $v_1=(r+1,r,a,1,r+2)$, and  $r+1=ad$.

There are quotient singularities $Q_2,Q_5$ of index $r,r+2$
respectively.
 We  take $g: Y \to X$  the  weighted
blowup with weights $w_2=\frac{1}{r+2}(d, 2d, 1, r-d+2, d)$ over
$Q_5$. Then
$$E^3 = \frac{2r+2}{ar(r+2)}, \quad F^3=\frac{(r+2)^2}{d(r-d+2)}, \quad \frak q=d, \quad \frak a=1.$$

 We  pick $D=f^{-1}_* \text{div}({x}_2)$. It is
easy to check that $E \cap D$ is  irreducible but non-reduced.
 We have $c_0=r, q_0=2d$, hence $c_0 -a q_0 <0$ and moreover
$$T(f,g,D)=\frac{1}{r+2}(-(2r+2)+\frac{2d}{ r-d+2})<0.$$
Therefore there exists a factoring diagram  by
Proposition \ref{2ray}.

$$\begin{CD}
Y @>{\dashrightarrow}>> Y^\sharp \\
@V{\frac{1}{r+2}}V{wt=w_2}V @V{{a-1}}V{wt=w'_2}V \\
Q_5 \in X @.  X^\sharp \ni Q^\sharp_4 \\
@V{{a}}V{wt=w_1  }V @V{1}V{wt=w'_1}V\\
X @>=>> X
\end{CD}$$

where $$ \begin{array}{l} w_1=v_1=(r+1,r,a,1,r+2),\\
w_2=\frac{1}{r+2}(d,2d,1,r-d+2,d), \\
w'_1=v_2=(d,d,1,1,d), \\
 w'_2=(r-d+1,r-d,2,a-1,1,r-d+2). \end{array}$$

\noindent {\bf Case IIa.} This contraction is described in
\cite[Theorem 1.1.(2)]{Kk02}. The local equation of $P \in W$ is
given by
$$(\varphi:x_1x_2+x_3^2+x_4^3=0) \subset \bC^4,$$
and $f$ is the weighted blowup with weights $v_1=(1,5,3,2)$.

There is a unique singularity $Q_2$ on $E$, which is a  quotient
singularities of index $5$. We take $g: Y \to X$  the weighted
blowup with weights $w_2=\frac{1}{5}(4,1,2,3)$  over $Q_2$.
 Thus $\frak q=1$, $\frak a=1$ and $\frac{\frak q}{5^3} F^3 =\frac{1}{30} < \frac{1}{5}$.
Therefore there exists a factoring diagram  by
Corollary \ref{2ray3}.

$$\begin{CD}
Y @>{\dashrightarrow}>> Y^\sharp \\
@V{\frac{1}{5}}V{wt=w_2}V @V{3}V{wt=w'_2}V \\
Q_2 \in X @.  X^\sharp \ni Q^\sharp_1 \\
@V{4}V{wt=w_1  }V @V{1}V{wt=w'_1}V\\
W @>=>> W
\end{CD}$$

where $$ \begin{array}{ll} w_1=v_1=(1,5,3,2), & w'_1=v_2=(1,1,1,1), \\
  w_2=\frac{1}{5}(4,1,2,3),  & w'_2=(1,4,2,1). \end{array}$$

\noindent{\bf Case IIb.} $f$ is of type e9 with discrepancy $2$.
This case was studied in \cite{GE}. We summarize some results in
\cite{GE}. There are two singularities $Q_1, Q_2$ of type
$\frac{1}{5}(1,1,-1)$ and $\frac{1}{3}(1,1,-1)$ respectively. Pick
any general elephant $S \in |-K_X|$, then $[S \cap E]=2[l]$, where
$l \cong \bP^1$ and $l$ passes through both $Q_1, Q_2$ \cite[Lemma
5.1]{GE}. We may assume that, near $Q_1$, $S={\rm div}(x)$,
$E={\rm div}(y^2)$ (after coordinate change) and $l=(x=y=0)$. Now
$E^3=\frac{1}{15}$ and $l \cdot E=\frac{-1}{15}$.

Let $g: Y \to X$ be the Kawamata blowup over $Q_1$ with weights
$\frac{1}{5}(1,1,4)$. One sees that $\frak q=2$, $\frak a=1$.
Notice that
$$2l_Y \cdot K_Y=2l \cdot K_X +\frac{2}{5^3}F^3 =
\frac{-2}{15}+\frac{2}{20} <0.$$ By Proposition \ref{2ray},
 there exists a factoring diagram.

\begin{diagram}
Y     &   & \rDashto &    &  Y^\sharp   \\
\dTo^{g}_{\frac{1}{5}} &       &      &   &  \dTo_{g^\sharp}      \\
 X       &        &      &  &  X^\sharp  \\
           & \rdTo_{f}^{a=2}  &      & \ldTo_{f^\sharp}^1  &        \\
           &        &  W   &        &
\end{diagram}
where $f^\sharp$ is a divisorial contraction with exceptional
divisor $F_{X^\sharp}$ and discrepancy $\frak a=1$.

\noindent{\bf Case IIc.} $f$ is of type e5 with discrepancy $2$.\\
There is only one singularity $Q \in X$, which is of type
$\frac{1}{7}(1,1,6)$. Let $g: Y \to W$ be the weighted blowup of
weights $\frac{1}{6}(1,1,6)$ over $Q$ and let $\mu: Z \to Y \to X
\ni Q$ be the economic resolution by further weighted blowups.
Clearly,
$$\left\{ \begin{array}{l} K_Z=\mu^* K_X + \sum_{j=1}^6 \frac{j}{7} F_j; \\
\mu^*E=E_Z+ \sum_{j=1}^6 \frac{q_j}{7} F_j, \end{array} \right.$$
for some $q_j$,  where $F_1=F$ is the exceptional divisor of $g$.
Hence
$$K_Z = \mu^* f^*K_W +2E_Z+ \sum_{j=1}^6 a_j  F_{j,Z}$$ with $a_j=\frac{2
q_j+j}{7} \in \bZ$.

Suppose that $E$ is given by $(\phi: \sum c_{\alpha \beta \gamma}
x^\alpha y^\beta z^\gamma=0)  \subset \bC^3/\frac{1}{7}(1,1,6)$
locally around $Q$. Then $$q_j:=\min\{ \alpha j + \beta j + \gamma
(7-j) | x^\alpha y^\beta z^\gamma \in \phi\} \ge \min\{j, 7-j\}.$$

By \cite{Ma}, there must exists a exceptional divisor with
discrepancy $1$ centering at $P\in W$. Since $Z \to W$ is a
Gorenstein partial resolution,  the exceptional with discrepancy
$1$ must appear in $Z$, that is, among $\{F_{j,Z}\}_{j=1,...,6}$.
One can verify that $F_1$ is the only exceptional divisor with
discrepancy $1$ and  $\frak{q}=q_1=3$. Hence $\frac{\frak{q}}{p^3}
F^3 =\frac{1}{14} < \frac{1}{7}$. By Corollary \ref{2ray3}, we
have a factoring diagram so that $f^\sharp: X^\sharp \to W$ is a
divisorial contraction contracting $F_{ X^\sharp}$ with
discrepancy $1$.

\noindent{\bf Case IId.} $f$ is of type e3 with discrepancy $3$.\\
There is only one singularity $Q \in X$, which is of type $cAx/4$
with axial weight $2$. More precisely, $Q \in X$ is given by
$$(\varphi: x^2+y^2+f(z,u)=0) \subset
\bC^4/\frac{1}{4}(1,3,1,2),$$ such that $u^3 \in \varphi$ and
$wt_{\frac{1}{4}(1,2)} f(z,u) = \frac{6}{4}$. By \cite[Theorem
7.4]{HaI}, there is a unique divisorial contraction $g: Y \to X$
over $Q$ with discrepancy $\frac{1}{4}$, which is the weighted
blowup of weights $\frac{1}{4}(5,3,1,2)$. Take economic resolution $\nu: Z \to Y$
over the unique higher index point, which is a quotient singularity of index $5$,
and let $\mu:= g \circ \nu: Z \to X$. Then  we ends up with
$$\left\{\begin{array}{l} K_Z=\mu^* K_X +\frac{1}{4}F+ \sum_{j=1}^4 \frac{b_j}{4} F_j; \\
\mu^*E=E_Z+ \frac{\frak{q}}{4} F+\sum_{j=1}^4 \frac{q_j}{4} F_j,
\end{array} \right.$$
where $F_j$ are  $\nu$-exceptional divisors and  $(b_1,b_2,b_3,b_4)=(2,2,3,4)$. Hence
$$K_Z = (f \circ \mu)^*K_W + \frak a F+\sum_{j=1}^4 a_j F_j,$$
where $\frak a=\frac{1+3\frak q}{4}$ and $a_j=\frac{b_j+3q_j}{4}$.
Since $a_j:= \frac{b_j+3 q_j}{4} >1$ for all $j$, it follows that
$F$ is the only exceptional divisor with discrepancy $1$ over $W$
and hence $\frak{q}=1$ and $\frak a =1$. Thus $\frac{\frak{q}}{p^3} F^3
=\frac{1}{20} < \frac{1}{4}$. By Corollary \ref{2ray3}, we have a
factoring diagram such that $f^\sharp: X^\sharp \to W$ is a
divisorial contraction with exceptional divisor $F_{ X^\sharp}$
and discrepancy $1$.

\noindent{\bf Case IIe.} $f$ is of type e2 with discrepancy $2$.\\
There is a unique higher index point $Q \in X$ of type $cA/r$ or
$cD/3$ with axial weight $2$.

\noindent{\bf Subcase 1.} $Q$ is of type $cD/3$.\\
Let $\mu: Z \to X$ be a common resolutions of $Q$ dominating all
divisorial contractions with minimal discrepancies over $Q$. We
have
$$K_Z =\mu^* K_X+ \sum_{j=1}^N \frac{1}{3}F_j+ \sum \frac{c_l}{3}G_l,$$
where $\{F_j\}_{j=1,...,N}$ is the set all all exceptional
divisors with discrepancy $\frac{1}{3}$ over $Q$ and $c_l \ge 2$.
Suppose that $\mu^*E= E_Z+ \sum\frac{q_j}{3}F_j+\sum
\frac{t_l}{3}G_l$, then
$$K_X =\mu^*f^* K_W + 2E_Z+ \sum_{j=1}^N a_j F_j +\sum b_l G_l,$$
where $a_j= \frac{2q_j+1}{3}$ and $b_l=\frac{2t_l+c_l}{3}>1$.
Since there exists an exceptional divisor with discrepancy $1$
over $P \in W$, we may assume that $a_1=1$.

By \cite[Section 9]{Ha1}, a $cD/3$ point can be classified as
$cD/$3-1, $cD/$3-2 and $cD/$3-3. Unless $Q \in X$ is  of type
$cD/$3-3 and Equation $*$ holds (cf. \cite[p.549]{Ha1}),
 we know that any
exceptional divisor with minimal discrepancy $\frac{1}{3}$ over a
$cD/3$ point is obtained by a divisorial contraction. Hence there
is a divisorial contraction $g: Y \to X$ with exceptional divisor
$F=F_1$ and discrepancy $\frac{1}{3}$. We thus have $\frak{q}=1$ and $\frak a =1$.

It is also straightforward to check that $\frac{\frak q}{3^3}F^3 =
\frac{1}{12}$ for any such divisorial contraction  with
discrepancy $\frac{1}{3}$.  By Corollary \ref{2ray3}, we have a
factoring diagram such that  $f^\sharp: X^\sharp \to W$ is a
divisorial contraction with exceptional divisor $F_{ X^\sharp}$
and discrepancy $1$.

In the remaining situation that $Q \in X$ is of type $cD/$3-3 and
Equation $*$ holds (cf. \cite[p.549]{Ha1}), then there is only one
divisorial contraction $g: Y \to X$, which is a weighted blowup
with weights $v_2=\frac{1}{3}(5,4,1,6)$. There is another
valuation with discrepancy $\frac{1}{3}$ given by the weighted
blowup with weights $v_1=\frac{1}{4}(2,4,1,3)$. We write
$K_Z=\mu^* K_X +\frac{1}{3}F_1+ \frac{1}{3}F_2+\sum
\frac{c_l}{3}G_l$, and
$$K_Z=\mu^*f^*K_W +2E_Z+ a_1F_1+a_2 F_2 + \sum b_l G_l,$$
where $F_i$ corresponds to the valuation with weights $v_i$ for
$i=1,2$.

  Let $(\phi=0) \subset \bC^3/\frac{1}{3}(2,1,1,0)$ be  the local equation of $E$ near $Q$. Since $a_1=1$, then $q_1=1$ and  $\frac{q_1}{3}=wt_{v_1}(\phi)=\frac{1}{3}$. One sees that $\phi$ contains $z$. It follows that $\frac{q_2}{3}=wt_{v_2}(\phi)=\frac{1}{3}$ and hence $\frak q=1$ and $\frak a=1$ holds.

Now we have $\frac{\frak q}{3^3}F^3 = \frac{1}{10}$.  By Corollary
\ref{2ray3} again, we have a factoring diagram such that
$f^\sharp: X^\sharp \to W$ is a divisorial contraction with
exceptional divisor $F_{ X^\sharp}$ and discrepancy $1$.

\noindent{\bf Subcase 2.} $Q$ is of type $cA/r$.\\
After coordinate changes, we may assume that local equation near
$Q$ is given by $(\varphi: xy+z^{tr}+u^2=0) \subset
\bC^4/\frac{1}{r}(1,-1,2,r)$ for some $t \ge 2$. Set $r=2k+1$. Let $Y
\to X$ be the weighted blowup with weights
$v_1:=\frac{1}{2k+1}(k+1,3k+1,1, 2k+1)$ with exceptional divisor $F$.
There are quotient singularities $R_1,R_2$ of index $k+1, 3k+1$.
Let $Z \to Y$ be the economic resolution of $R_1, R_2$.
Then we have
$$\begin{array}{ll} K_Z= & \mu^* K_X+ \frac{1}{2k+1}F + \sum_{j=1}^k \frac{2j}{2k+1} F_j + \\
  &\sum_{i=1}^{k}( \frac{2i+1}{2k+1} G_{0i}+ \frac{2i}{2k+1}G_{1i}+\frac{2i-1}{2k+1}G_{2i}). \end{array}$$
More explicitly, the resolution over $R_1$ is obtained by weighted
blowups of weights $\frac{1}{k+1}(j,2k+2-2j, j, k+1-j)$ for $1 \le
j \le k$. Over $Q$ these weights corresponds to vectors
$\frac{1}{2k+1}(j, 4k+2-j, 2j, 2k+1)$. Similarly, the resolution
over $R_2$ is obtained by weighted blowups of weights
$\frac{1}{3k+1}(2i, 3k+1-i, 3i, i), \frac{1}{3k+1}(2k+2i, 2k+1-i,
3i-1, k+i)$, and $\frac{1}{3k+1}(4k+2i, k+1-i, 3i-2, 2k+i)$ for $1
\le i \le k$. Over $Q$, these weights corresponds to vectors
$$\left\{ \begin{array}{l}\frac{1}{2k+1}(k+1+i,3k+1-i, 2i+1 , 2k+1), \\
\frac{1}{2k+1}(2k+1+i,2k+1-i, 2i , 2k+1),\\
\frac{1}{2k+1}(3k+1+i,k+1-i, 2i-1 , 2k+1). \end{array} \right.$$
for $1 \le i \le k$ respectively.

Suppose that $E$ is given by $(\phi: \sum c_{\alpha \beta \gamma
\delta }x^\alpha y^\beta z^\gamma u^\delta=0)  \subset
\bC^4/\frac{1}{r}(1,-1,2, r)$ locally around $Q$. We write
$\mu^*E= E_Z+ \frac{\frak q}{2k+1} F+ \sum_{j=1}^k
\frac{q_j}{2k+1} F_j + \sum_{i=1}^k (\frac{t_{0i}}{2k+1}
G_{0i}+\frac{t_{1i}}{2k+1} G_{1i}+\frac{t_{2i}}{2k+1} G_{2i})$ and
hence
$$K_Z=\mu^*f^*K_W+2E_Z+\frak a F+ \sum_{j=1}^k a_j F_j + \sum_{i=1}^k (b_{0i} G_{0i}+b_{1i} G_{1i}+b_{2i} G_{2i}),$$
with $\frak a:=\frac{2 \frak q+1}{2k+1}$, $a_j:=\frac{2
q_j+2j}{2k+1}$ , $b_{0i}:=\frac{2 t_{0i} +2i+1}{2k+1}$,
$b_{1i}:=\frac{2 t_{1i} +2i}{2k+1}$, $b_{2i}:=\frac{2 t_{2i}
+2i-1}{2k+1}$. There exists an exceptional divisor with
discrepancy $1$. Hence either $\frak a$, $b_{0i}$ or $b_{2i}=1$
for some $i$ because $a_j$ and $b_{1i}$ are even.

\noindent {\bf Claim.} $\frak a=1$.\\
Suppose that $b_{0i}=1$ for some $i$. Then $t_{0i}=k-i$. Since
$$t_{01}=\min \{ \alpha (k+1+i) + \beta (3k+1-i) + \gamma (2i+1)
+\delta (2k+1)| x^\alpha y^\beta z^\gamma u^\delta \in \phi\}.$$
It follows that $\phi$ contains  $z^\gamma$ with $\gamma (2i+1) =
k-i$. Hence $$\frac{\frak q}{2k+1} =wt_{v_1} \phi \le
\frac{k-i}{2k+1} \le \frac{k-1}{2k+1}$$ and $\frak a<1$, a
contradiction.

Suppose that $b_{2i}=1$ for some $i$. Then similarly, one sees
that $\phi$ contains $z^\gamma$ with $\gamma(2i-1)=k-i+1$. This
leads to the same contradiction unless $b_{21}=1$ and $\phi$
contains $z^k$. It follows that $\frak q=k$ and $\frak a=1$ in
this situation.

Now $\frac{\frak q }{(2k+1)^3} F^3 = \frac{2k}{(k+1)(3k+1)(2k+1)}
< \frac{1}{2k+1}$. By Corollary \ref{2ray3}, there is a factoring
diagram such that $f^\sharp$ is a divisorial contraction with
discrepancy $\frak a=1$.

\noindent{\bf Case IIf.} $f$ is of type e1 with discrepancy $2$.\\
In this case, there is a unique higher point $Q$ of type
$\frac{1}{r}(1,-1,4)$.

\noindent{\bf Subcase 1.} $r=4k+3$.\\
 Let $Y \to X$ be the Kawamata blowup along $Q$ with weights $\frac{1}{4k+3}(k+1, 3k+2, 1)$.
 Suppose that  the local equation of $E$ near $Q$ is given by $(\phi: \sum c_{\alpha \beta \gamma}x^\alpha y^\beta z^\gamma=0)$.
Let $\mu: Z \to X$ be the economic resolution over $Q$, which factors through $Y$. Then we have
$$\left\{\begin{array}{l} K_Z=\mu^*K_X+ \sum_{j=1}^{4k+2} \frac{j}{4k+3}F_j; \\ \mu^*E= E_Z +\sum_{j=1}^{4k+2} \frac{q_j}{4k+3} F_j, \end{array} \right.$$
where $F_1=F$ and  $$q_j:=\min\{ \alpha \overline{(k+1)j}+\beta
\overline{(3k+2)j} + \gamma j |x^\alpha y^\beta z^\gamma \in \phi
\}.$$ We have $K_Z=g^*f^*K_W + 2E_Z+ \sum_{j=1}^{4k+2} a_j F_j$ with $a_j
= \frac{2q_j+j}{4k+3} \in \bZ$. Note that $a_j \equiv j
\ (\text{mod } 2)$ and $a_j=1$ for some $j$.

\noindent{\bf Claim}. $a_1 \le 3$.\\
Suppose on the contrary that $a_1 \ge 5$. For all monomial
$x^\alpha y^\beta z^\gamma \in \phi$, we have $$ q_1=\alpha (k+1)
+ \beta (3k+2)+\gamma \ge 10k+7. \eqno{\dagger}$$ If $a_j=1$ for
some $j$, then

$$q_{j}=\left\{ \begin{array}{lll}2k-2s+1 &=(k+s+1)\alpha+(3k-s+2)\beta+(4s+1)\gamma, & \text{if } j=4s+1;\\
                             2k-2s&=(3k+s+3)\alpha+(k-s)\beta+(4s+3)\gamma, & \text{if } j=4s+3,   \end{array} \right.$$ for some $x^\alpha y^\beta z^\gamma \in \phi$, which is a contradiction to $\dagger$. This completes the proof of the Claim.

Notice  that if $a_1=3$, i.e. $q_1=6k+4$, then $y^2 \in \phi$ and $a_j=1$ if and only if $j=4s+3$ with $s<k$. In this case, there are exactly $k-1$ exceptional divisors with discrepancy $1$. Hence $k \ge 2$ in this situation.
Also, if $a_1=1$, then $q_1=2k+1$.
Thus in any event, $$\frac{\frak q}{(4k+3)^3}F^3 = \frac{2 \frak q}{(k+1)(3k+2)(4k+3)} \le \frac{4}{3(4k+3)}.$$
For any $l \subset S \cap E$, one has $l \cdot E \ge \frac{1}{4k+3}$ and hence
$l \cdot K_X \le \frac{-2}{4k+3}$. Therefore,
 $l_{Y} \cdot K_Y < 0$ for all $i$. Hence there exists a factoring diagram by Corollary \ref{2ray2}. The resulting divisorial contraction $f^\sharp: X^\sharp \to W$ is a divisorial contraction with discrepancy $1$ or $3$.

\noindent{\bf Subcase 2.} $r=4k+1$.\\
Similarly, let $Y \to X$ be the Kawamata blowup along $Q$ with weights
$\frac{1}{4k+1}(3k+1, k, 1)$ and
 $\mu: Z \to X$ be the economic resolution over $Q$,
which factors through $Y$.

Thus we have $K_Z=g^*f^*K_W + 2E_Z+ \sum_{j=1}^{4k} a_j F_j$ with $a_j = \frac{2q_j+j}{4k+1} \in \bZ$ and $$q_j:=\min\{ \alpha \overline{(3k+1)j}+\beta \overline{kj} + \gamma j |x^\alpha y^\beta z^\gamma \in \phi \}.$$
Note that $a_j \equiv j \ (\text{mod } 2)$ and $a_j=1$ for some $j$.

\noindent {\bf Claim}. $a_1 =1$.\\
Suppose on the other hand that $a_1 \ge 3$.
For all monomial $x^\alpha y^\beta z^\gamma \in \phi$, we have $$ q_1=\alpha (3k+1) + \beta k+\gamma \ge 6k+1. \eqno{\dagger}$$
Suppose that  $a_j=1$, it is straightforward to see that

$$q_{j}=\left\{ \begin{array}{lll}2k-2s+1&=(k+s)\alpha+(3k-s+1)\beta+(4s-1)\gamma, & \text{if } j=4s-1;\\
                             2k-2s &=(3k+s+1)\alpha+(k-s)\beta+(4s+1)\gamma, & \text{if } j=4s+1,   \end{array} \right.$$ for some $x^\alpha y^\beta z^\gamma \in \phi$, which is a contradiction to $\dagger$. The Claim now follows.

Now $\frak a=a_1=1$, $\frak {q} =2k$ and thus $$\frac{\frak q}{(4k+1)^3}F^3 = \frac{4}{(3k+1)(4k+1)} \le \frac{1}{4k+1}.$$
For any $l \subset S \cap E$, one has $l \cdot E \ge \frac{1}{4k+1}$ and hence
$l \cdot K_X \le \frac{-2}{4k+1}$. Therefore,
 $l_{Y} \cdot K_Y < 0$ for all $i$. Hence there exists a factoring diagram by Corollary \ref{2ray2}. The resulting map $f^\sharp: X^\sharp \to W$ is a divisorial contraction with discrepancy $1$.

\noindent{\bf Case IIg.} $f$ is of type e1 with discrepancy $4$.\\
In this case, there is a unique higher index point $Q$ of type
$\frac{1}{r}(1,-1,8)$. One can work out this case similar to Case
IIf.

\noindent{\bf Subcase 1.} $r=8k+7$.\\
Let $Y \to X$ be the Kawamata blowup along $Q$ with weights
$\frac{1}{8k+7}(k+1, 7k+6, 1)$ and $\mu: Z \to X$ be the economic resolution
over $Q$, which factors through $Y$. Suppose that  the local equation
of $E$ near $Q$ is given by $(\phi: \sum c_{\alpha \beta \gamma} x^\alpha y^\beta
z^\gamma=0)$.
Thus we have $K_Z=\mu^*f^*K_W + 4E_Z+ \sum_{j=1}^{8k+6} a_j F_j$ with
$a_j = \frac{4q_j+j}{8k+7} \in \bZ$ and $$q_j:=\min\{ \alpha \overline{(k+1)j}+\beta
\overline{(7k+6)j} + \gamma j |x^\alpha y^\beta z^\gamma \in \phi
\}.$$ Note that $a_j \equiv -j
\ (\text{mod } 4)$ and $a_j=1$ for some $j$.

\noindent{\bf Claim}. $a_1=3$ or $7$. \footnote{if $a_1=7$,  then $y^2 \in \phi$ and $a_j=1$ if and only if $j=8s+3$ with $s<k$. In this case, there are exactly $k-1$ exceptional divisors with discrepancy $1$.}\\
Suppose on the contrary that $a_1 \ge 11$. For all monomial
$x^\alpha y^\beta z^\gamma \in \phi$, we have $$ q_1 \ge \alpha
(k+1) + \beta (7k+6) +\gamma \ge 22k+19. \eqno{\dagger}$$ Suppose that 
$a_j=1$, it is straightforward to see that

$$q_{j}=\left\{ \begin{array}{lll} 
                             2k-2s+1&=(3k+s+3)\alpha+(5k-s+4)\beta+(8s+3)\gamma, & \text{if } j=8s+3;\\
                             2k-2s&=(7k+s+1)\alpha+(k-s)\beta+(8s+7)\gamma, & \text{if } j=8s+7,\\
                              \end{array} \right.$$ for some $x^\alpha y^\beta z^\gamma \in \phi$, which is a contradiction to $\dagger$. The Claim now follows.

Now $\frak {q} \le 14k+12$ and thus $$\frac{\frak q}{(8k+7)^3}F^3 = \frac{2 \frak q}{(k+1)(7k+6)(8k+7)} \le \frac{4}{(k+1)(8k+7)}.$$
For any $l_i \subset S \cap E$, one has $l_i \cdot E \ge \frac{1}{8k+7}$ and hence
$l_i \cdot K_X \le \frac{-4}{8k+7}$. Therefore,
 $l_{i,Y} \cdot K_Y \le 0$ for all $i$ and strictly $<0$ for some $i$. Hence there exists a factoring diagram by Proposition \ref{2ray2}. The resulting map $f^\sharp: X^\sharp \to W$ is a divisorial contraction with discrepancy $3$ or $7$.

\noindent{\bf Subcase 2.} $r=8k+5$.\footnote{if $a_1=5$,  then $y^2 \in \phi$ and $a_j=1$ if and only if $j=8s+5$ with $s<k$. In this case, there are exactly $k-1$ exceptional divisors with discrepancy $1$.}\\
Similar argument shows that $a_1=1$ or $5$ ( since $a_1 \equiv 1 \ (\text{mod } 4)$) and there exists a factoring diagram by Corollary \ref{2ray2}. The resulting map $f^\sharp: X^\sharp \to W$ is a divisorial contraction with discrepancy $1$ or $5$.

\noindent{\bf Subcase 3.} $r=8k+3$.\\
Similar argument shows that $a_1=3$ ( since $a_1 \equiv -1 \ (\text{mod } 4)$) and there exists a factoring diagram by Proposition \ref{2ray2}. The resulting map $f^\sharp: X^\sharp \to W$ is a divisorial contraction with discrepancy $3$.

\noindent{\bf Subcase 4.} $r=8k+1$.\\
Similar argument shows that $a_1=1$ ( since $a_1 \equiv 1 \ (\text{mod } 4)$) and there exists a factoring diagram by Proposition \ref{2ray2}. The resulting map $f^\sharp: X^\sharp \to W$ is a divisorial contraction with discrepancy $3$.

\section{proof of the main theorem}
\begin{proof}
We prove by induction on depth and discrepancies.

{\bf 1.} Suppose first that $\dep(X)=0$, that is, $X$ has at worst Gorenstein terminal singularities.
By the classification of Mori and Cutkosky \cite{Mo82, Cu}, $f$ can not be a flipping contraction.

If $f: X \to W$ is a divisorial contraction to a point
then $f$  is a divisorial contraction with minimal discrepancy (cf. \cite{Mo82, Cu}).

If $f: X \to W$ be a divisorial contraction to a curve, then $f$ is a blowup along a lci curve in a smooth neighborhood by the
classification of Mori and Cutkosky again.
By Proposition \ref{lci}, $f$ is factorizable.

{\bf 2.} Let $f: X \to W$ be a divisorial contraction to a curve
$\Gamma$ with $\dep(X)=d>0$.  By \cite{CH}, there is a factoring
diagram
\begin{diagram}
Y     &   & \rDashto &    &  Y^\sharp   \\
\dTo^{g} &       &      &   &  \dTo_{g^\sharp}      \\
 X       &        &      &  &  X^\sharp  \\
           & \rdTo_{f}  &      & \ldTo_{f^\sharp}  &        \\
           &        &  W   &        &
\end{diagram}

satisfying:
\begin{enumerate}
\item $Y \to X$ is a divisorial contraction to a highest index
point of index $r>1$ with discrepancy $\frac{1}{r}$;

 \item $Y \to
Y^\sharp$ is a sequence of flips and flops;

\item $g^\sharp: Y^\sharp \to X^\sharp$ is divisorial contraction to the proper transform of $\Gamma$;

\item $f^\sharp$ is a divisorial contraction to a point.
\end{enumerate}

Note that $\dep(Y)=d-1$, and  $\dep(Y^\sharp) \le \dep(Y)=d-1$.
Therefore by Proposition \ref{depth},
$$\dep(X^\sharp) \le \min(0, \dep(Y^\sharp)-1) <d.$$ It follows
that  $X \to W$ can be factored into $$X \dashrightarrow Y
\dashrightarrow  Y^\sharp \to X^\sharp \to W$$ so that each map is
factorizable by induction on depth.

{\bf 3.} Let $f: X \to W$ be a flipping  contraction.  By \cite{CH},
there is a factoring diagram as above so that $f^\sharp:
X^\sharp=X^+ \to W$ is the flipped contraction. Similarly, each
map of $$X \dashrightarrow Y \dashrightarrow  Y^\sharp \to
X^\sharp=X^+$$ is factorizable by induction on depth.

{\bf 4.} Let $f: X \to W$ be a divisorial contraction to a point
$P\in W$ of index $r$  with $\dep(X)=d$ and discrepancy
$\frac{1}{r}$. Nothing to do.

{\bf 5.} Let $f: X \to W$ be a divisorial contraction to a point
$P\in W$ of index $r>1$ with $\dep(X)=d$ and discrepancy
$\frac{a}{r} > \frac{1}{r}$. By \cite{Pisa}, there is a factoring
diagram satisfying:
\begin{enumerate}
\item $Y \to X$ is a divisorial contraction to a highest index
point of index $r>1$ with discrepancy $\frac{1}{r}$;

\item $Y \to Y^\sharp$ is a sequence of flips and flops;

\item $f^\sharp$ is a divisorial contraction with discrepancy $\frac{a'}{r} < \frac{a}{r}$;

\item $g^\sharp$ is divisorial contraction to a point $Q$ of index $r$ with  discrepancy $\frac{a''}{r} < \frac{a}{r}$ and $a''+a'=a$ if $P\in W$ is not  of type $cE/2$;

\item  $g^\sharp$ is divisorial contraction to a point $Q$ of index $3$ with  discrepancy $\frac{1}{3} $ if $P\in W$ is  of type $cE/2$.
\end{enumerate}

Notice  that $\dep(Y^\sharp) \le \dep(Y) = d-1$ and
$\dep(X^\sharp) \le \dep(Y^\sharp)+1 \le d$. By induction on
depth, both $Y \dashrightarrow Y^\sharp$ and $Y^\sharp \to
X^\sharp$ are factorizable. If $\dep(X^\sharp)< \dep(X)$, then we
are done by induction. If $\dep(X^\sharp)= \dep(X)$, then we may
proceed by induction on $a$ which measures the discrepancy.

{\bf 6.} Let $f: X \to W$ be a divisorial contraction to a point
$P\in W$ of index $1$ with $\dep(X)=d$ and discrepancy $a >1$.

{\bf 6.1} If $P \in W$ is a non-singular point, then by the study
of Case Ia, $f$ is factorizable by induction on $a$.

{\bf 6.2} If $P \in W$ is of type $cA$. By the studies in Case Ib,
IIa, and IId, there exists a factoring diagram such that
$f^\sharp: X^\sharp \to W$ has discrepancy $a_1<a$ (Case Ib) or
$1$ (Case IIa, IId). Moreover $\dep(X^\sharp) \le d$. Therefore,
$f^\sharp$ is factorizable by induction on discrepancy $a$ hence
so is $f: X \to W$ because $Y \dashrightarrow Y^\sharp \to X^\sharp$
having $\dep <d$.

{\bf 6.3} If $P \in W$ is of type $cD$ or $cE$ and the
discrepancy $a$ is odd.\\
This could be Case Ic, Id, IId. There exists a factoring diagram
such that $f^\sharp: X^\sharp \to W$ has discrepancy $a_2<a$ (Case
Ic) or $1$ (Case Id, IId). Similarly $f$ is factorizable by
induction on $a$ and on depth.

{\bf 6.4} If $P \in W$ is of type $cD$ or $cE$ and the discrepancy
$a$ is even. This could be Case  Id, IIb, IIc, IIe, IIf, and IIg.
There exists a factoring diagram such that $f^\sharp: X^\sharp \to
W$ has odd discrepancy $a_1$ (Case IIf, IIg) or $1$ (other cases).
Therefore, $f$ is factorizable by $6.3$ and induction on depth.

\end{proof}


\end{document}